\documentclass[12pt]{amsart}

\usepackage{amsfonts,amsmath,amsthm,amssymb}

\def\gp#1{\langle#1\rangle}
\def\m1{^{-1}}

\newtheorem{theorem}{Theorem}
\newtheorem{lemma}{Lemma}
\newtheorem{corollary}{Corollary}

\newtheorem*{problem}{Problem}

\title[Free subgroups]
{Free subgroups in    group rings}
\author{Victor A.~Bovdi}

\address{
\texttt{VICTOR BOVDI}
\newline
Department of Math. Sciences, UAE University, Al-Ain,
\newline United Arab Emirates} \email{vbovdi@gmail.com}

\thanks{The research was supported by FAPESP (Brazil) Process N: 2014/18318-7.}


\begin{document}

\begin{abstract}
Let $V(\mathbb{K}G)$ be the normalized   group of  units of the group ring $\mathbb{K}G$ of a non-Dedekind  group $G$ with  nontrivial torsion part $t(G)$ over the integral domain $\mathbb{K}$.
We give a simple method for constructing free objects in $V(\mathbb{K}G)$.
In particular, we show  that $V(\mathbb{K}G)$  always contains  the free product $C_n \star C_n$ of two finite cyclic groups. We construct
examples of subgroups in  $V(\mathbb{K}G)$ which  are either  cyclic extensions of a non-abelian free group or $C_n \star C_n$.\end{abstract}
\maketitle

Let $V(\mathbb{K}G)$ be the group of normalized units of the group ring $\mathbb{K}G$ of a group  $G$  with  nontrivial torsion part (i.e. the set of elements of finite order) $t(G)$ over the integral domain $\mathbb{K}$.

In their classic paper, B.~Hartley and P.~F. Pickel  (see \cite{Hartley_Pickel}) proved
that if $G$ is  a finite non-Dedekind  group, then $V(\mathbb{Z}G)$   contains a free group of rank $2$. After the publication of this result, several authors have  studied the following problem:  When do  two special units generate   a free group of rank $2$?  A fundamental result was published by  A.~Salwa \cite{Salwa}, who proved that two noncommuting  unipotent elements $\{1+x, 1+x^*\}$ of $\mathbb{Z}G$ always generate   a free group of rank $2$, where $x$ is a nilpotent element and $*$ is the classical involution of $\mathbb{Z}G$.

As an example for the  unipotent element $1+x$ in Salwa's  paper, one can take the bicyclic unit $u_{a,b}=1+(a-1)b\widehat{a}\in V(\mathbb{K}G)$ (see the notation
below). This example raises the following problem.
\begin{problem}
When does a unit $w\in V(\mathbb{K}G)$ exist with the property that $\gp{u_{a,b}, w}$ contains  a free subgroup of rank $2$ for  fixed  $a,b\in G$?
\end{problem}
Similar problems  were  studied  in several papers by A.~Dooms,  J.~Gon\-calves, R.M.~Guralnick, E.~Jespers,
V.~Jim\'enez, L.~Margolis, Z.~Marci\-niak, D.~Passman, A.~del Rio, M.~Ruiz and  S.~ Sehgal.
As the literature of this problem is quite voluminous, we do not cite
particular papers,
as it would be impossible to do justice to the researchers of this field.
Nevertheless the reader
can easily find the  papers relevant to this problem.

Theorem 2 of our  paper provides an explicit answer to this question, proving that  for fixed $a,b\in G$, it is enough to choose $w=a^k\in G$.

For the group of units of group rings of several groups Theorem 1  provides
a simple alternative proof of the classical result of B.~Hartley
and P.~F. Pickel. Moreover, we prove   that the  group ring $\mathbb{K}G$ of a non-Dedekind group  $G$
which has at least one non-normal finite cyclic subgroup  of order  $n$  always contains the free product  $C_n \star C_n$ as a subgroup. Furthermore, this subgroup $C_n \star C_n$  can  be generated normally by a single element. Several problems in group  theory and the theory of small dimensional topology can be reduced to the question whether a given group can be normally generate by a single element.  In particular  the Relation Gap problem, Wall's $D2$ Conjecture, the Kervaire Conjecture, Wiegold's Problem, Short's Conjecture and the Scott-Wiegold Conjecture (for example, see \cite{Kourovka_book}, Questions 5.52, 5.53 and 17.94).

Finally,  in Lemma 2,  we introduce and study a new family of torsion and non-torsion units in  $V(\mathbb{K}G)$. This  lemma might have some significance in itself.

We denote by $C_n$ and $C_\infty$  the cyclic group of  order $n$ and of
infinite order, respectively.
If $A,B\leq G$ are  subgroups of  $G$, then we denote by $A\star B$ the free product of these subgroups.
Denote  the normalizer of a subgroup $H$ in $G$ by ${\mathfrak N}_G(H)$. If $|a|$ is the order of  $a\in t(G)$, then we put  $\widehat{a}=\sum_{i=1}^{|a|}a^i\in \mathbb{Z}G$. If $x=\sum_{g\in G}\alpha_gg\in \mathbb{K}G$,
then  $supp(x)$ denotes  the set $\{g\in G \mid \alpha_g\not=0\}$. The $gcd$ of the natural numbers $k$ and $l$ is denoted by $(k,l)$.

Let $\mathfrak{C}$ be a  class of groups. The group $G$ is called  {\it residual} for $\mathfrak{C}$,
if for each $g\in G\setminus  \{1\}$, there exists a normal subgroup $N\lhd G$ such that
$g\not\in N$ and $G/N \in \mathfrak{C}$.

\bigskip

Our main results are the following.

\begin{theorem}\label{T:1}
Let $\mathbb{K}$ be an integral domain and let $G$ be a group which has at least one non-normal
finite cyclic subgroup $\gp{a}$ of order $|a|$. Let   $b\in G\setminus {\mathfrak N}_G(\gp{a})$.
Additionally let  $M$ be the smallest  integer  with the properties that
$2\leq M\leq |a|$ and $b\in {\mathfrak N}_G(\gp{a^{M}})$.

Let $1\leq k< |a|$ with the property  that $b\not \in {\mathfrak N}_G(\gp{a^{k}})$. The elements  $\frak{u}_k=a^k+(a-1)b\widehat{a}$ and
$\frak{z}_k=w\m1\frak{u}_k w$ are units in $\mathbb{K}G$, where $w=1+(a-1)b\widehat{a}\in V(\mathbb{K}G)$ and the following  hold:
\begin{itemize}
\item[(i)]
if $(k,|a|)=1$, then $\gp{\frak{u}_k, \frak{z}_k}\cong C_{|a|}\star C_{|a|}$;

\item[(ii)]
if $(k,|a|)\not=1$,  then $\gp{\frak{u}_k, \frak{z}_k}\cong C_{s}\star C_{s}$, where  $s=\frac{|a|}{(k,|a|)}$  if   $(k,M)=1$ and $M\neq |a|$, otherwise
\[
\gp{\frak{u}_k, \frak{z}_k}\cong C_{\infty}\star C_{\infty}.
\]

\end{itemize}
\end{theorem}

\begin{theorem}\label{T:2}
Let $\mathbb{K}$ be an integral domain and let $G$ be a group which has at least one non-normal
finite cyclic subgroup $\gp{a}$  of order $|a|$. Let   $b\in G\setminus {\mathfrak N}_G(\gp{a})$.
Additionally let  $M$ be the smallest  integer with the properties that
$2\leq M\leq |a|$ and $b\in {\mathfrak N}_G(\gp{a^{M}})$.

Let  $1\leq k< |a|$  with the property  that $b\not \in {\mathfrak N}_G(\gp{a^{k}})$.  Put $H_k=\gp{\; 1+(a-1)b\widehat{a},\;   a^k\;} \leq V(\mathbb{K}G).
$ Then the following  hold:
\begin{itemize}
\item[(ii)] if $(k,|a|)=1$  then $H_k$
is a cyclic extension of $C_{|a|}\star C_{|a|}$;

\item[(ii)]
if $(k,|a|)\not=1$ and  $b\not\in {\mathfrak N}_G(\gp{a^k})$, then  $H_k$
is a cyclic extension of $C_{s}\star C_{s}$, where $s=\frac{|a|}{(k,|a|)}$ if
$M\not=|a|$ and  $(k,M)\not=1$, otherwise it is a cyclic extension of a non-abelian free group.
\end{itemize}
Moreover,  in these cases $H_k$  is a residually torsion-free nilpotent group.
\end{theorem}

\begin{corollary}\label{C:1}
Let $\mathbb{K}$ be an integral domain and let $G$ be a group which has at least one non-normal
finite cyclic subgroup $\gp{a}$ of order $|a|$. Then the group of units $U(\mathbb{K}G)$ of the group ring $\mathbb{K}G$ always contains the free product $C_{|a|} \star C_{|a|}$ as a subgroup which is normally generated by a single element.
\end{corollary}

To prove our results we first recall the following  well known facts.

\smallskip

Let $\gp{a\mid a^n=1}$ be a cyclic group. Define  the  homomorphism
\[
\psi: \mathbb{K}[x] \to \mathbb{K}\gp{a}\cong \mathbb{K}[x]/\gp{x^n-1}.
\]
We often use the fact that for
any $w\in \mathbb{K}\gp{a}$ there exists a polynomial  $f(x)\in \mathbb{K}[x]$ of degree
$deg(f(x))<n$, such that $\psi(f(x))=w$.

\begin{lemma}\label{L:1} (\cite{Bovdi_book}, Proposition 2.7, p.9)
Let $H$ be  a subgroup of a group $G$ and let $K$ be a ring. The
left  annihilator  $L$
in $KG$ of the right  ideal
\[
\mathfrak{I}_r(H)=\gp{h-1\mid h\in H,\;  h\not=1}
\]
 is different  from zero if and only if $H$ is finite.
If $H$ is finite, then
\[
L=KG(\Sigma_{h\in H}h ).
\]
\end{lemma}

\bigskip
Now we are able to prove the following  lemma.
\begin{lemma}\label{L:2}
Let $G$ be a group which has at least one non-normal
finite cyclic subgroup $\gp{a}$ and  let  $b\in G\setminus {\mathfrak N}_G(\gp{a})$.
Additionally let  $M$ be the smallest  integer  with the properties that
$2\leq M\leq |a|$ and $b\in {\mathfrak N}_G(\gp{a^{M}})$.

For an integer $k$  with  $1\leq  k\leq |a|$,
the element $\frak{u}_k=a^k+(a-1)b\widehat{a}$ is a unit in $\mathbb{K}G$ and the following hold:
\begin{itemize}
\item[(i)] if $(k,|a|)=1$, then $\frak{u}_k$ has  finite order $|a|$;
\item[(ii)] if $(k,|a|)\not=1$ then $\frak{u}_k$ has  finite order $\frac{|a|}{(k,|a|)}$ if   $(k,M)=1$ and $M\neq |a|$, otherwise it has  infinite order.
\end{itemize}
\end{lemma}
\begin{proof}
Clearly $1\not=w=1+(a-1)b\widehat{a}\in V(\mathbb{K}G)$ and the element
\[
\frak{u}_k=a^k+(a-1)b\widehat{a}=\big(1+(a-1)b\widehat{a}\big)a^k, \qquad (1\leq  k\leq  |a|)
\]
is also a unit of $\mathbb{K}G$, because  $\frak{u}_k\m1=a^{-k}-a^{-k}(a-1)b\widehat{a}$. Moreover
\begin{equation}\label{E:1}
\frak{u}_k^i=a^{ik}+\Big(\sum_{j=0}^{i-1}a^{jk}\Big)(a-1)b\widehat{a}.
\end{equation}
(i) First let  $(k,|a|)=1$.  By Lemma \ref{L:1} we have
\[
\Big(\sum_{j=0}^{|a|-1}a^{jk}\Big)(a-1)=\Big(\sum_{j=0}^{|a|-1}a^{j}\Big)(a-1)=0
\]
so by (\ref{E:1}) the orders of the elements $\frak{u}_k$ and $a$ coincide.

%\newpage

(ii) In the sequel we  assume that  $(k,|a|)\not=1$. Define the following  integers:\quad
$s_k=\frac{|a|}{(k,|a|)}$\quad  and\quad  $t_k=\frac{s_k}{\frac{M}{(M,k)}}=\frac{|a|\cdot (k,M)}{M\cdot (k,|a|)}$ .

First, let $b\not\in {\mathfrak N}_G(\gp{a^M})$   for any $2\leq M\leq |a|-1$. Clearly
\[
\textstyle
\Big(\sum_{j=0}^{{s_k}-1}a^{jk}\Big)(a-1)\not=0,
\]
by Lemma \ref{L:1}.
Hence $\frak{u}_k$ has infinite order, because from (\ref{E:1}) and from the gene\-ralized Berman-Higman's theorem (see \cite{Artamonov_Bovdi}, {\bf 1.2}, p.5 or \cite{Bovdi_trace}) it  follows that
\[
\textstyle
tr\big(\frak{u}_k^{s_k}\big)=tr\Big(1+\big(\sum_{j=0}^{{s_k}-1}a^{jk}\big)(a-1)b\widehat{a}\Big)=1\not=0.
\]
Let  $M$ be the smallest  integer with the properties that  $2\leq M\leq |a|$
 and $b\in {\mathfrak N}_G(\gp{a^{M}})$.
Since $(k,m)=1$, the set $\{\; ik\; \mid\;  0\leq i\leq s_k-1\; \}$ is a cyclic sequence $mod\; M$, so
\[
\{\; ik\mid 0\leq i\leq s_k-1\; \}\pmod{M} = \{\; j\mid 0\leq j\leq M-1\; \}
\]
and\quad $\frac{|\{\; ik\;  \mid \;  0\leq i\leq s_k-1\; \}|}{|\{\; j\; \mid \; 0\leq j\leq M-1\; \}|}=t_k$. It follows that
\[
\begin{split}
(\frak{u}_k)^{s_k}&=a^{ks_k}+\textstyle\Big(\sum_{i=0}^{s_{k}-1}a^{ik}\Big)(a-1)b\widehat{a}\\
&=1+\textstyle t_k\Big(\sum_{j=0}^{M-1}a^j\Big)(a-1)b\widehat{a}\\
&=1+\textstyle t_k(a^{M}-1)b\widehat{a}=1,
\end{split}
\]
so the order of  $\frak{u}_k$ is equal  to $s_k=\frac{|a|}{(k,|a|)}$.

Now let $(k,M)=k_M\not=1$.
Clearly  the set $\{\; ik\; \mid\;  0\leq i<s_k\; \}$ is a cyclic sequence $mod\; M$, so
\[
\textstyle
\{\; ik\mid 0\leq i\leq s_k-1\; \}\pmod{M} = \{\; jk_M\mid 0\leq j\leq \frac{M}{(M,k)}-1\; \}
\]
and \quad $\textstyle
\frac{|\{\; ik\; \mid \; 0\leq i\leq s_k-1\; \}|}{|\{\; j\; \mid \; 0\leq j\leq \frac{M}{(M,k)}-1\; \}|}=t_k$.\quad
It follows that
\[
\begin{split}
(\frak{u}_k)^{s_k}&=a^{ks_k}+\textstyle\Big(\sum_{i=0}^{s_{k}-1}a^{ik}\Big)(a-1)b\widehat{a}\\
&=1+\textstyle t_k\Big(\sum_{j=0}^{\frac{M}{(M,k)}-1}a^{j(k,M)}\Big)(a-1)b\widehat{a}\not=1,\\
\end{split}
\]
so $\frak{u}_k$ has infinite order, because again from the gene\-ralized Berman-Higman's theorem (see \cite{Artamonov_Bovdi}, {\bf 1.2}, p.5 or \cite{Bovdi_trace})  we have
$tr\big(\frak{u}_k^{s_k}\big)=1\not=0$. \end{proof}

\begin{lemma}\label{L:3}
Let $G=\gp{a}$ be a finite cyclic  group  and  let $1\leq  k< |a|$.
If\quad   $\Delta_k(x)=\sum_{i=0}^{k-1}x^i$,\;  $\Delta_{-k}(x)=\sum_{i=0}^{|a|-k-1}x^i\in \mathbb{Z}[x]$\; and $\widehat{x}=\sum_{i=0}^{|a|-1}x^i\in \mathbb{Z}[x]$,  then   for any integers  $j,l,s \geq 0$ and  $\beta\in\mathbb{Z}$  the following conditions in $\mathbb{K}G$ hold:
\begin{equation}\label{E:2}
0\not= \Delta_{\mp k}^{j}(a) \Delta_{\pm k}^{l}(a) (a-1)^s\not=\beta\widehat{a}.
\end{equation}
\end{lemma}
\begin{proof}
Clearly $\mathbb{K}\gp{a}\cong \mathbb{K}[x]/\gp{x^n-1}$, where   $n=|a|$. Let   $\xi\in\mathbb{C}$ be a primitive root of unity of order $n$.
Obviously  $deg\big(\Delta_{\pm k}(x)\big)\leq n-2$.

First of all $\Delta_{\pm k}(\xi)\not=0$. Indeed, if $\Delta_{\pm k}(\xi)=0$ then
\[
0=\Delta_{\pm k}(\xi)(\xi-1)=\xi^m-1,\qquad\quad   (m\in \{|a|-k, k\})
\]
a contradiction  because $m<n$ and $\xi-1\not=0$.

Now, if \;  $\Delta_{\mp k}^{j}(x)\Delta_{\pm k}^{l}(x)(x-1)^{s}=\beta\widehat{x}$\quad  for some integers $j,l,s\geq 0$ and $\beta\in \mathbb{Z}$, then
\;  $\Delta_{\mp k}^{j}(\xi) \cdot\Delta_{\pm k}^{l}(\xi)\cdot (\xi-1)^{s}=\beta\widehat{\xi}=0\in\mathbb{C}$,\quad
so $\Delta_{\pm k}(\xi)=0$, a contradiction. Consequently,
\[
\Delta_{\mp k}^{j}(x)\Delta_{\pm k}^{l}(x)(x-1)^{s}\not=\beta\widehat{x}
\]
for any integers $j,l, s \geq 0$ and  $\beta\in\mathbb{Z}$.

Now the rest of  (\ref{E:2}) follows  trivially from  Lemma \ref{L:1}, because
\[
Ann_l\big(\mathfrak{I}_r(\gp{a})\big)\ni \beta\widehat{a}\not=\Delta_{\mp k}^{j}(a)\Delta_{\pm k}^{l}(a)(a-1)^{s-1}.
\]\end{proof}

\begin{lemma}\label{L:4}
Let $G$ be a group which has at least one non-normal
finite cyclic subgroup $\gp{a}$ and let   $b\in G\setminus {\mathfrak N}_G(\gp{a})$. Let  $1\leq  k < |a|$,  such that
$b\not\in {\mathfrak N}_G(\gp{a^k})$.

Set     $\Delta_k(x)=\sum_{i=0}^{k-1}x^i\in \mathbb{Z}[x]$ and $\Delta_{-k}(x)=\sum_{i=0}^{|a|-k-1}x^i\in \mathbb{Z}[x]$.
Define the following elements of $\mathbb{K}G$:
\[
\begin{aligned}
x_{+}&=(a-1)\big(\Delta_k(a)+b\widehat{a}\big),&\qquad y_{+}&=(a-1)\big(\Delta_k(a)+a^k b\widehat{a}\big), \\
x_{-}&=(a-1)\big(\Delta_{-k}(a)-a^{-k} b\widehat{a}\big),&\qquad y_{-}&=(a-1)\big(\Delta_{-k}(a)- b\widehat{a}\big).
\end{aligned}
\]
If $z_\alpha, z_\beta\in \{x_{\pm}, y_{\pm}\}$ then $z_{\alpha}z_{\beta}=\Delta_{\alpha k}(a)(a-1)z_{\beta}\not=0$, where  $\alpha\in\{\pm \}$.
\end{lemma}
\begin{proof} Since $b\not\in {\mathfrak N}_G(\gp{a})$, we have $x_{\alpha}\not=0$ and $y_{\alpha}\not=0$. Now
\[
x_{\alpha}(a-1)=\Delta_{\alpha k}(a)(a-1)^2,\quad
y_{\alpha}(a-1)=\Delta_{\alpha k}(a)(a-1)^2,
\]
so it can be  easily   verified that  $z_{\alpha}z_{\beta}=\Delta_{\alpha k}(a)(a-1)z_{\beta}$.

If $\Delta_{\alpha k}(a)(a-1)z_{\beta}=0$, then       $\Delta_{\alpha k}(a)\Delta_{\pm k}(a)(a-1)^2\not=0$  by  (\ref{E:2}), so
\[
\Delta_{\alpha k}(a)\Delta_{\pm k}(a)(a-1)^2\not=\pm \Delta_{\alpha k}(a)(a-1)^2a^{\pm ks} b\widehat{a}, \quad (s\in\{0,1\})
\]
because the supports of these elements are different, a contradiction. Hence
$z_{\alpha}z_{\beta}=\Delta_{\alpha k}(a)(a-1)z_{\beta}\not=0$, where
$z_\alpha, z_\beta\in \{x_{\pm}, y_{\pm}\}$. \end{proof}

\bigskip

\begin{proof}[Proof Theorem 1]
Let  $1\leq  k < |a|$ with the property  that
$b\not\in {\mathfrak N}_G(\gp{a^k})$. Set  $\omega=|a|-k$,    $\Delta_k(x)=\sum_{i=0}^{k-1}x^i\in \mathbb{Z}[x]$\quad  and $\Delta_{-k}(x)=\sum_{i=0}^{\omega-1}x^i\in \mathbb{Z}[x]$. It is easy to see that
\[
\begin{split}
\frak{u}_k=a^k+(a-1)b\widehat{a}&=1+(a-1)\Big(\Delta_k(a)+b\widehat{a}\Big);\\
\frak{z}_k=w\m1\frak{u}_k w&=1+(a-1)\Big(\Delta_k(a)+a^kb\widehat{a}\Big); \\
%\end{split}
%\]\[
%\begin{split}
\frak{u}_k^{-1}=a^{\omega}-a^{\omega}(a-1)b\widehat{a}&=1+(a-1)\Big(\Delta_{-k}(a)-a^{-k}b\widehat{a}\Big);\\
\frak{z}_k^{-1}=a^{\omega}-(a-1)b\widehat{a}&=1+(a-1)\Big(\Delta_{-k}(a)-b\widehat{a}\Big).\\
\end{split}
\]
Here   $\frak{u}_k^{\pm 1}=1+x_{\pm}$, $\frak{z}_k^{\pm 1}=1+y_{\pm}$ and
$x_{\pm}, y_{\pm}$ are from    Lemma \ref{L:4}.

\smallskip

Let  $1\leq k < |a|$ and let $m \in \mathbb{N}$, such that  $1\leq m<|a|$. Define
\begin{equation}\label{E:3}
\textstyle
F_{m,k}(x)=\sum_{i=1}^{m}{m\choose i}\Delta_{k}(x)^{i-1}(x-1)^{i-1}\in \mathbb{Z}[x].
\end{equation}
Let $\tau k$ (this is a symbol, not a product) denote a natural number from
$\{ k, |a|-k\}$. Obviously   $F_{m,\tau k}(x)\not=0$ and  from (\ref{E:3}) we get
\begin{equation}\label{E:4}
\begin{split}
1+F_{m, \tau k}(x)\Delta_{\tau k}(x)(x-1)&=\Big(1+\Delta_{\tau k}(x)(x-1)\Big)^{m}\\
&=(1+x^{\tau k}-1)^{m}\\
&=x^{{(\tau k)}m}.
\end{split}
\end{equation}
If  $m \in \mathbb{Z}$  and $\alpha, \gamma\in \{ \pm \}$, then  using induction
on $|m|\geq 2$,  from Lemma \ref{L:4},    (\ref{E:3}) and from the fact
${m \choose i} +{m \choose {i-1}}={{m+1} \choose i}$ we have
\begin{equation}\label{E:5}
(1+x_\alpha)^{m}=1+F_{|m|, \tau k}(a)x_\gamma,\quad  (1+y_\alpha)^{m}=1+F_{|m|, \tau k}(a)y_\gamma,
\end{equation}
where the symbol $\tau k=k$  for  $m\geq 0$ and otherwise $\tau k=|a|- k$. Moreover  the symbol $\gamma=\alpha$  for  $m\geq 0$ and otherwise $\gamma$ is the opposite sign of $\alpha$. This yields that  if
$l_i\in \{x_{+}, y_{+}\}$ and $\alpha_i\in \mathbb{Z}$  then
\begin{equation}\label{E:6}
\textstyle
(1+l_1)^{\alpha_1}\times \cdots \times (1+l_s)^{\alpha_s}=\prod_{i=1}^s(1+z_i)^{\beta_i}
\end{equation}
where $z_i\in \{x_{\pm}, y_{\pm}\}$ and $\beta_i=|\alpha_i|$ for all $i=1,\ldots, s$.

Any word   from the set  $\{\frak{u}_k, \frak{z}_k\}$ can  be reduced to the following form  $P_s=t_1^{\alpha_1}\cdots t_s^{\alpha_s}$, where
$t_i=1+l_i\in\{\frak{u}_k, \frak{z}_k\}$, $l_i\in \{x_{+}, y_{+}\}$, $t_i\not=t_{i+1}$ and $s>0$. Moreover, if $\frak{u}_k$ has infinite order, then   $\alpha_i\in \mathbb{Z}$, otherwise  $|\alpha_i|<|a|$.

If $|\alpha_i|\geq |a|$ for some $i$, then there exist integers $l> 0$ and $\alpha_i'$, such that $|\alpha_i|=l(|a|-1) +\alpha_i'$ and $0\leq \alpha_i'<|a|-1$. Put $\varepsilon=sign(\alpha_i)$.  Then  we write
$t_i^{\alpha_i}=t_i^{\varepsilon |\alpha_i|}=
\underbrace{ t_i^{\varepsilon(|a|-1)}\cdots t_i^{\varepsilon(|a|-1)}}_lt_i^{\varepsilon \alpha_i'}$.
Furthermore if
$(|\alpha_i|,|a|)\not=1$, then we write \quad $t_i^{\alpha_i}=t_i^{\varepsilon (|a|-1)}t_i^{\varepsilon 1}$.

Applying these reductions,  we can assume that   for a fixed $s$, the word  $P_s$ can be written as  $P_s=t_1^{\alpha_1}\cdots t_s^{\alpha_s}$ (with a new value of $s$), where $t_i\in \{1+x_+, 1+y_+\}$,  $1\leq |\alpha_i|<|a|$ and
$(|\alpha_i|,|a|)=1$  for each $1\leq i\leq s$. Then
\[
\begin{aligned}
\textstyle
P_s=\prod_{i=1}^s t_i^{\alpha_i}=\prod_{i=1}^s(1+l_i)^{\alpha_i}&=\prod_{i=1}^s(1+z_i)^{\beta_i}& \text{by  }(\ref{E:6})\\
&\textstyle
=\prod_{i=1}^s(1+f_iz_i),\quad   &\text{by  }(\ref{E:5})\\
\end{aligned}
\]
where $\beta_i=|\alpha_i|$, $z_i\in \{x_{\pm}, y_{\pm}\}$ and   $f_i=F_{\beta_i,\tau_i k}(a)$. Note that
the symbol $\tau_i k=k$ if $\alpha_i\geq 0$ and $\tau_i k=|a|-k$, otherwise. Denote $\Delta_{\tau_i k}(a)$ by $\Delta_i$.

Put $T_j(a)=f_j-\beta_j=F_{\beta_j,\tau_j k}(a)-\beta_j$. Similarly as  in Lemma \ref{L:4},
$z_iT_j(a)=(a-1)\Delta_iT_j(a)$, so using this fact and  Lemma \ref{L:4}, we get
\begin{equation}\label{E:7}
\begin{split}
z_i f_jz_j&=z_i(\beta_j+T_j(a))z_j=\beta_jz_iz_j+(a-1)\Delta_iT_j(a)z_j\\
&\textstyle=\beta_j(a-1)\Delta_iz_j+(a-1)\Delta_iT_j(a)z_j\\
&\textstyle=\Delta_i(a-1)\cdot f_jz_j.\\
\end{split}
\end{equation}
Now using (\ref{E:7})  and  Lemma \ref{L:4} we obtain that
\begin{equation}\label{E:8}
\begin{aligned}
P_s&\textstyle
=1+f_1z_1+\sum_{i=2}^{s}\Big(\prod_{j=1}^{i-1}\big(1+(a-1)\Delta_jf_j\big)\Big)f_iz_i&\\
&\textstyle
=1+f_1z_1+\sum_{i=2}^{s}\big( a^{\pi_i} \big)f_iz_i&\text{by }(\ref{E:4})\\
&\textstyle
=1+\sum_{i=1}^{s}a^{\pi_i}f_iz_i,&\\
\end{aligned}
\end{equation}
where $0\leq \pi_i<|a|$ and put $\pi_1=0$. In the sequel the exact value of $\pi_i$ (except that  $\pi_1=0$) is not important for us.

Since  $z_i\in \{x_{\pm}, y_{\pm}\}$, from the last equation   we obtain that
\begin{equation}\label{E:9}
P_s=1+ G_1(a)x_{-}+G_2(a)x_{+}+ G_3(a)y_{-}+G_4(a)y_{+},
\end{equation}
where $G_j(a)=\sum_{l\in I_j}a^{\pi_l} f_l$,  $(j=1,\ldots,4)$ and  $I_1, I_2,I_3, I_4$ form a pairwise
distinct partition of the set $\{1, \ldots,s\}$.
Clearly at least one (say $x_{-}$)  of the elements from $\{x_{\pm}, y_{\pm}\}$ has to appear on the right side.

Let us prove that $G_1(a)\not=0$. Clearly $f_l\not=0$ for all $l\in I_1$, because
\[
f_l\cdot \Delta_{-k}(a)(a-1)=\big(a^{\beta_l}\big)^{-k}-1\not=0
\]
by (\ref{E:4}) as $(\beta_l, |a|)=1$ and $|a|-k<|a|$.

In the sequel  for  the polynomial $F_{m,|a|-k}(x)$ from  (\ref{E:3})
we allow that $m\leq |a|$ and
the condition $(m,|a|)=1$ is not required.
For each  $a^{\pi_l}$ there is a  $\beta_j=\beta_j(l)$ such that $a^{\pi_l}=1+\Delta_{-k}(a-1)f_j$ (see (\ref{E:4})).

Note that each $1<\beta_i< |a|$ (see the definition of $P_s$). Now put $\beta_n=|a|$, where $n=s+1$.
Moreover if $\beta_n\leq \beta_i+\beta_j< 2\beta_n$, then $\beta_i+\beta_j=\beta_n+\beta_k$ for some $1\leq \beta_k<\beta_n$
(of course the condition  $(\beta_k,|a|)=1$  is not required) and, consequently,  (see (\ref{E:5})) the following equation holds
\begin{equation}\label{E:10}
(1+f_ix_{-})(1+f_jx_{-})=(1+f_kx_{-})(1+f_nx_{-}),
\end{equation}
where $f_i=F_{\beta_i,-k}(a)$. Then
\[
\begin{aligned}
a^{\pi_l}f_ix_{-}&=\big(1+f_j\Delta_{-k}(a-1)\big)f_ix_{-}=(f_i+f_ix_{-}f_j)x_{-}&\text{by (\ref{E:7})}\\
&=(1+f_ix_{-})(1+f_jx_{-})-f_jx_{-}-1&\\
&=(1+f_kx_{-})(1+f_nx_{-})-f_jx_{-}-1&\text{by (\ref{E:10})}\\
&=f_k\big(1+\Delta_{-k}(a-1)f_n\big)x_{-}+f_nx_{-}-f_jx_{-}&\text{by (\ref{E:7})}\\
&=f_k\big((a^{|a|})^{-k}-1\big)x_{-}+f_nx_{-}-f_jx_{-}&\text{by (\ref{E:4})}\\
&=(f_n-f_j)x_{-}.&\\
\end{aligned}
\]
Now using the last equation, (\ref{E:4}) and the well known equation
\[
x^{i+j}-1=(x^{i}-1)(x^{j}-1)+(x^{i}-1)+(x^{j}-1)
\]
it is easy to see  that
\begin{equation}\label{E:11}
\begin{split}
G_1&(a)\cdot \Delta_{-k}(a)(a-1)=\\
&= \textstyle\Big(\sum_{\pi_l+\beta_l(n-k)\geq n} a^{\pi_l} f_l+
\sum_{\pi_l+\beta_l(n-k)< n} a^{\pi_l} f_l\Big)\Delta_{-k}(a)(a-1) \\
&\textstyle=\big(\lambda_n f_n-\sum_{j}\lambda_jf_j\big)\Delta_{-k}(a)(a-1)\\
&\quad\textstyle +\sum_{\pi_l+\beta_l(n-k)< n}\Big[(a^{\pi_l}-1)\big((a^{\beta_l})^{-k}-1\big)
+\big((a^{\beta_l})^{-k}-1\big)\Big]\\
&\textstyle=\sum_{i=1}^{|a|-1}\gamma_i(a-1)^i\not=0, \qquad\quad  (\lambda_j, \gamma_i\in \mathbb{Z})
\end{split}
\end{equation}
because the elements $(a-1)^i$ for $1\leq i<|a|$ form a basis of the augmentation ideal $\omega(\mathbb{K}\gp{a})$
of the group ring $\mathbb{K}\gp{a}$ (see \cite{Bovdi_book}), and at least $a^{\pi_1}=1$ in (\ref{E:8}).
Consequently $G_1(a)\not=0$.

Now let us prove  that $G_1(a)x_{-}\not=0$.    Indeed, if  $G_1(a)x_-=0$,  then
\[
\textstyle
\big(\sum_{l\in I_1}a^{\pi_l} f_l\big)\Delta_{-k}(a) (a-1)=\big(\sum_{l\in I_1}a^{\pi_l} f_l\big) (a-1)a^{-k} b\widehat{a},
\]
which is impossible,  because the left hand side of   this equation is non-zero by (\ref{E:11}) and
the supports of the left and right hand sides  are different.

Consequently,  $G_1(a) x_{-}\not=0$ (Similarly it is easy to apply this technique to prove that
either    $G_2(a) x_{+}\not=0$ or $G_3(a) y_{-}\not=0$ or $G_4(a) y_{+}\not=0$, as  one of $\{x_{+}, y_{\pm}\}$
necessarily  appears in  (\ref{E:9})).
Moreover, applying a  similar argument  to the nonzero part of
\[
M(a)=G_1(a)x_{-}+G_2(a)x_{+}+ G_3(a)y_{-}+G_4(a)y_{+},
\]
we can   show  as in (\ref{E:9}) that
\[
\begin{split}
\Big((G_1&(a)+G_3(a))\cdot \Delta_{-k}(a)+(G_2(a)+G_4(a))\cdot \Delta_{k}(a)\Big)(a-1)\\
&\textstyle=\sum_{i=1}^{|a|-1}\gamma_i(a-1)^i\not=0, \qquad\quad  (\gamma_i\in \mathbb{Z})
\end{split}
\]
because the elements $(a-1)^i$ for $1\leq i<|a|$ form a basis of the augmentation ideal $\omega(\mathbb{K}\gp{a})$
of the group ring $\mathbb{K}\gp{a}$ (see \cite{Bovdi_book}), and at least $a^{\pi_1}=1$ in (\ref{E:8}).
Consequently  $M(a)\not=0$.
This yield that   $P_s\not=1$  for any $s>0$, which means that  $\gp{\frak{u}_k, \frak{z}_k}$ is a free group.

The rest of the proof follows from Lemma \ref{L:2}.\end{proof}

\smallskip

Note that if $G=D_{2p}$ is the dihedral group of order $2p$ ($p$ is a prime), then  part (i) of Theorem \ref{T:1}
does not imply that $V(\mathbb{K}G)$ contains a free group as a subgroup.

\smallskip

\begin{proof}[Proof of Theorem 2]
In \cite{Dison_Riley},  W.~Dison and T.~R.~Riley  introduced a family of one-relator
groups
\begin{equation}\label{E:12}
\mathfrak{H}_r(x,y)= \gp{x,y \mid (x,\underbrace{y,y,\ldots,y}_r)=1},\qquad (r\geq 1)
\end{equation}
that  are called  {\it Hydra} groups. These groups are   cyclic extension of a non-abelian
free group. In \cite{Baumslag_Mikhailov},  G.~Baumslag and R.~Mikhailov   proved that the
Hydra groups (similarly  to free groups) are residually torsion-free nilpotent.

Let $G$ be a group which has at least one non-normal
finite subgroup $\gp{a}$ and let  $b\in G\setminus {\mathfrak N}_G(\gp{a})$. If  $1\leq k< |a|$,
then  by Lemma \ref{L:2}
\[
w=1+(a-1)b\widehat{a}\qquad  \text{and}\qquad  \frak{u}_k=a^k+(a-1)b\widehat{a}=wa^k
\]
are nontrivial units in $\mathbb{K}G$  and
$\gp{w,a^k}= \gp{w,w a^k}= \gp{w, \frak{u}_k}$.

Using a straightforward calculation, we have
\[
\begin{split}
(\frak{u}_k,w)&
\textstyle=(\frak{u}_k\m1 w\m1)\cdot (\frak{u}_kw)\\
&\textstyle=\Big(a^{-k}-2(a^{1-k}-a^{-k})b\widehat{a}\Big)\Big(a^k+(a^{k}+1)(a-1)b\widehat{a}\Big)\\
&
\textstyle=1+\big(a^{-k}-a^{1-k}+ a-1\big)b\widehat{a}\not=1,\\
\end{split}
\]
because $b\notin {\mathfrak N}_G(\gp{a})$ and  $b\notin {\mathfrak N}_G(\gp{a^k})$, respectively.
\bigskip

Finally, it  is easy to check that
\begin{equation}\label{E:13}
(\frak{u}_k,w,w)=1.
\footnote{Note that this relation between the units   $\frak{u}_k$ and $w$ holds
in an arbitrary group ring $KG$ over a ring $K$ with $char(K)\geq 0$.}
\end{equation}

%\bigskip
Now we are able to prove our result.

\noindent
First let the elements  $w, \frak{u}_k \in V(\mathbb{K}G)$ have infinite order (see Lemma \ref{L:2}).  Denote  a free group of rank  $2$ by $\mathfrak{F}_2$. Since  the Hydra group $\mathfrak{H}_2(x, y)$ (see (\ref{E:12})) is the third member of  the following exact sequence
\[
1 \rightarrow \gp{x, x^y}\cong \mathfrak{F}_2 \rightarrow \mathfrak{H}_2(x, y) \rightarrow \gp{y}\cong C_\infty \rightarrow 1,
\]
by (\ref{E:13})  and by Theorem \ref{T:1}(i) we obtain that
\[
H_k=\gp{w,a^k}=\gp{w, \frak{u}_k(=wa^k)}\cong \mathfrak{H}_2(\frak{u}_k, w).
\]
\noindent
Now let   $\frak{u}_k \in V(\mathbb{K}G)$ be of  finite  order $|\frak{u}_k|$ (see Lemma \ref{L:2}). Similarly to the previous case, by (\ref{E:13})  we obtain that
\[
H_k=\gp{w,a^k}=\gp{ \frak{u}_k, w}\cong \gp{x,y \mid x^{|\frak{u}_k|}=1, (x,y,y)=1},
\]
so the proof of our theorem     follows from Lemma \ref{L:2}, Theorem \ref{T:1} and  Theorem 1 of \cite{Baumslag_Mikhailov}.
\end{proof}

\begin{proof}[Proof of the Corollary.]
Let $G$ be a group which has at least one non-normal
finite cyclic subgroup $\gp{a}$ and let   $b\in G\setminus {\mathfrak N}_G(\gp{a})$.
Put $\frak{u}_k=a+(a-1)b\widehat{a}$ (i.e. $k=1$). Clearly  $|\frak{u}_k|=|a|$ by Lemma \ref{L:2}, so the result follows from Theorem \ref{T:1}(i).
\end{proof}

Our result motivates the following

\begin{problem}
When does $V(\mathbb{K}G)$ contain  the Hydra group $\mathfrak{H}_r(x,y)$ as a subgroup for $r\geq 3$?
\end{problem}

Finally, note that   if the group  $G$  has at least one non-normal
finite subgroup $\gp{a}$ and   $g,h\in G\setminus {\mathfrak N}_G(\gp{a})$, then in \cite{Bovdi_free_II},
using the technique of this paper,  it was proved that
 the  group $\gp{\; a^i+(a-1)g\widehat{a},\; a^j+(a-1)h\widehat{a}\;}$ is isomorphic either to
$C_\infty\star C_\infty$ or to  $C_{m}\star C_\infty$ or to $C_{m}\star C_{l}$, where  $1\leq  i\not= j < |a|$ and $l,m\in \{|a|, \frac{|a|}{(|a|,i)}, \frac{|a|}{(|a|,j)} \}$.

The author would like to express his gratitude to Eric Jespers  who very carefully read this paper
and made valuable remarks.

\end{document}